\documentclass[11pt]{amsart}%
\usepackage{color}
\usepackage{amsmath}
\usepackage{amssymb}
\usepackage{amsfonts}
\usepackage[latin1]{inputenc}
\usepackage{graphicx}%
\providecommand{\U}[1]{\protect\rule{.1in}{.1in}}
\providecommand{\U}[1]{\protect\rule{.1in}{.1in}}

\DeclareMathSymbol{\subsetneqq}{\mathbin}{AMSb}{36}

\theoremstyle{plain}
\numberwithin{equation}{section}
\newtheorem{theorem}{Theorem}[section]

\newtheorem{lemma}{Lemma}[section]
\newtheorem{proposition}{Proposition}[section]
\newtheorem{definition}{Definition}[section]
\newtheorem{remark}{Remark}[section]

\begin{document}
\title[On the convergence of a perturbed one dimensional Mann's process]
{On the convergence of a perturbed one dimensional Mann's process}%
\author{Ramzi May}%
\address{Department of Mathematics and Statistics, College of Science, King Faisal University, Al-Ahsa, Kingdom of Saudi Arabia}
% ORCID iD: http://orcid.org/0000-0002-1121-5854}
\email{rmay@kfu.edu.sa}
%\thanks{This work is supported by the Deanship of Scientific Research at King Faisal University under the Project​}
.
\vskip 0.2cm
\date{ July 13 , 2026}

% ----------------------------------------------------------------
\begin{abstract} We study a perturbed version of Mann's iterative process (DMP), defined by
\[
x_{n+1} = (1 - \theta_n)x_n + \theta_n f(x_n) + r_n,
\]
where $f: [0,1] \to [0,1]$ is a continuous function, $\{\theta_n\} \subset [0,1]$ is a given sequence, and $\{r_n\} $ represents an error term. We prove that if the sequence $\{\theta_n\} $ converges sufficiently slowly to zero and the error term $ r_n $ is suitably small at infinity, then any sequence $\{x_n\} \subset [0,1] $ generated by this process converges to a fixed point of  $f$. In addition, we investigate the asymptotic behavior of the trajectories $ x(t) $ as $ t \to \infty$ for a continuous-time version of the process (DMP). We emphasize the parallels between the discrete and the continuous dynamics. Furthermore, through numerical experiments, we analyze the influence of the sequence $\{\theta_n\} $ and the error terms on the stability and the convergence rate of the the discrete and the continuous processes. Notably, we observe that the (DMP) algorithm, when affected by stochastic and relatively large error terms, can outperform the bisection method in efficiently identifying the fixed point set of the function  $f$.\end{abstract}
\maketitle

\section{Introduction and main results}

Let $f:[0,1]\to [0,1]$ be a continuous function. A simple application of the intermediate value theorem ensures that the function $f$ has at least one fixed point. Since $f(0)(f(1)-1)\leq0$ then an application of the deterministic bisection method to the function $g(x)=f(x)-x$ leads to a numerical approximation of a particular fixed point of $f$ with any prefixed tolerance $\epsilon>0$. In order to discover more fixed points of $f$, it is worth to try to construct iterative processes that converge to a fixed point of $f$ whatever the value of the initial data $x_0$ in $[0,1]$. We might initially consider using the classical Banach iterative process $x_{n+1}=f(x_n),n=0,1,\cdots$. But in general this process does not converge. To see this it suffices to consider the case where  $f(x)=1-x$ and the initial data $x_0\neq\frac{1}{2}$. In fact, in this case the sequence $\{x_n\}$ is given explicitly by $x_n=\frac{1}{2}+(-1)^n(x_0-\frac{1}{2})$. In a pioneer paper \cite{Man}, W. Robert Mann introduced the following iterative process
\begin{equation}
\label{Mp}
  x_{n+1}=\frac{n}{n+1}x_n+\frac{1}{n+1}f(x_n).
\end{equation}
He established the following result.
\begin{theorem}
\label{Th01}
Let $f:[0,1]\to [0,1]$ be a continuous function. If $f$  has a
unique fixed point $p\in [0,1]$, then, for any initial data $%
x_{0}\in [0,1],$ the sequence $\{x_{n}\}$ generated by the process (\ref{Mp})
converges towards $p$.
\end{theorem}
Later, R.L. Franks and R.P. Marzec, in a brief and elegant paper \cite{FM}, showed that the uniqueness condition for the fixed point of the function
$f$ in the previous theorem is not necessary. Precisely, they established the following result.
\begin{theorem}
\label{Th02}
If $f:[0,1]\to [0,1]$ is a continuous function then, for any initial data $%
x_{0}\in [0,1],$ the sequence $\{x_{n}\}$ generated by the process (\ref{Mp})
converges towards a fixed point of $f$
\end{theorem}
The first objective of this paper is to examine the impact of small computational errors on the asymptotic behavior of the iterative process (\ref{Mp}). In this context, we establish the following result.
\begin{theorem}
\label{Th1}
Let $f:[0,1]\to [0,1]$ be a continuous function and $\{\theta_n\}\in[0,1]$ a given real sequence. Let  $\{x_{n}\}\in[0,1]$ be a sequence satisfying the perturbed iterative process
\begin{equation}
\label{PMp}
  x_{n+1}=(1-\theta_n) x_n+\theta_n f(x_n)+r_n,
\end{equation}
where $\{r_n\}$ is a real sequence. If
\begin{enumerate}
  \item $\theta _{n}\rightarrow 0$ as $n\rightarrow \infty $,
  \item $\sum_{n=0}^{\infty }\theta _{n}$ diverges,
  \item $\frac{r_{n}}{\theta _{n}}\rightarrow 0$ as $n\rightarrow
\infty $,
\item $\sum_{n=0}^{\infty }r_{n}$ converges,
\end{enumerate}
then the sequence $\{x_{n}\}$ converges towards a fixed point of $f$.
\end{theorem}
\begin{remark}
Let $f:[0,1]\to [0,1]$ be a nondecreasing function. It is well-known that $f$ has at least one fixed point even in the absence of continuity. This naturally raises the question of whether iterative schemes such as the Banach process $ x_{n+1} = f(x_n)$, the algorithm~\eqref{Mp}, or its general and perturbed form~\eqref{PMp} converge to a fixed point of $f $, regardless of the choice of the initial value $x_0 \in [0,1]$. To date, the author has not arrived at a solution of this problem.
\end{remark}
\begin{remark}
Variate version of the Mann's process (DMP) are used in many scientific research works to approximate the fixed points of different kinds of function defined on closed and convex subsets of a Hilbert spaces or a uniformly convex Banach spaces. See for instance \cite{Halp,Kri,Mou,Rei,Xu}.
\end{remark}
Inspired by the works \cite{JH}, \cite{Hal}, \cite{Har}, \cite{May} and \cite{SBC}, we also study the asymptotic behavior of the trajectories $x(t)$  of a continuous version
of the discrete dynamical system (\ref{PMp}). We obtain the following result.
\begin{theorem}
\label{Th2}
Let $f:[0,1]\to [0,1]$ and $\theta:[0,\infty)\to [0,\infty)$ be two continuous functions and let $x:[0,\infty)\to [0,1]$ be a continuous differentiable function that satisfies  the perturbed differential equation:
\begin{equation}
\label{CMp}
  x'(t)+\theta(t)x(t)=\theta(t)f(x(t))+r(t), t\geq0,
\end{equation}
where $r:[0,\infty)\to \mathbb{R} $ is a continuous function. If
\begin{enumerate}
  \item $\lim_{t\rightarrow \infty }\frac{r(t)}{\theta (t)}=0,$
  \item $\int_{0}^{\infty }\theta (t)dt$ diverges,
  \item $\int_{0}^{\infty }r(t)dt$ converges,
\end{enumerate}
then  $x(t)$ converges towards a fixed point of $f$ as $t\rightarrow\infty$.
\end{theorem}
 The rest of paper is organized as follows: In the next section, we provide a detailed proof of Theorem \ref{Th1}. The third section is devoted to the proof of Theorem \ref{Th2}. In the fourth section, we study through  some numerical experiments  the effect of the sequence $\{\theta_n\}$ and the stochastic error term $\{r_n\}$ on the stability and the rate of convergence of the discrete process (\ref{PMp}). We also study the effect of the asymptotic behavior of the function $\theta(t)$ and the error term $r(t)$ on the speed of the convergence of the trajectories of the continuous process (\ref{CMp}). The last section of the paper is devoted to a complete and detailed proof of Mann's convergence original result (Theorem \ref{Th01}).

\section{On the convergence of the perturbed discrete dynamical system (\ref{PMp})}
In this section, we provide a detailed proof of Theorem \ref{Th1}. The proof is greatly inspired by the original  paper \cite{FM}. It essentially relies on the
classical notion of the omega limit set of a real sequence. We recall here
briefly this notion and some of its main properties.

\begin{definition}
Let $\{x_{n}\}$ be a real sequence. The omega limit set of the sequence $%
\{x_{n}\}$ is the set $\omega (\{x_{n}\})$ of real numbers $z$ such
that $z=\lim_{n\rightarrow \infty }x_{n_{k}}$ for some subsequence $%
\{x_{n_{k}}\}$ of the sequence $\{x_{n}\}.$
\end{definition}

\begin{lemma}
\label{l1}
Let $\{x_{n}\}$ be a real sequence. Then

\begin{enumerate}
  \item $\omega (\{x_{n}\})=\cap _{n\in \mathbb{N}}\overline{%
\{x_{m}:m\geq n\}}^{\mathbb{R}}$ where $\overline{%
\{x_{m}:m\geq n\}}^{\mathbb{R}}$ denotes the closure of the set $\{x_{m}:m\geq n\}$ in $\mathbb{R}$.
\item If $\{x_{n}\}$ is bounded then $\omega (\{x_{n}\})$ is a nonempty compact
subset of $\mathbb{R}$.
\item If $\{x_{n}\}$ is bounded then $\{x_{n}\}$ converges in $\mathbb{R%
}$ if and only if the set $\omega (\{x_{n}\})$ is reduced to a single
element.
\item If $\lim_{n\rightarrow \infty }x_{n+1}-x_{n}=0$ then $\omega
(\{x_{n}\})$ is a connected subset of $\mathbb{R}$.
\end{enumerate}
\end{lemma}
The proof of this lemma is classical and easy.
Now we are in position to prove our first  main result (Theorem \ref{Th1}).
\begin{proof}
Since the sequence $\{x_{n}\}$ is in $[0,1]$ and satisfies
\[
x_{n+1}-x_{n}=\theta _{n}[f(x_{n})-x_{n}]+r_{n}\rightarrow 0\text{ as }%
n\rightarrow \infty ,
\]%
then, in view of the previous lemma, there exist two real number $0\leq a\leq b\leq 1$
such that $\omega (\{x_{n}\})=[a,b].$ Let us suppose that $a<b.$ We claim
that in this case $f(w)=w$ for any $w\in (a,b).$ Let $w\in (a,b).$ Suppose
for the sake of contradiction that $f(w)>w.$ Then thanks to the continuity of $f$ there exist two positive
real numbers $\delta $ and $\varepsilon $ such that $[w-\delta ,w+\delta
]\subset (a,b)$ and $f(x)\geq x+\varepsilon $ for any $x\in \lbrack w-\delta
,w+\delta ].$ On the other hand, there exists $n_{0}\in \mathbb{N}$ such that for any $n\geq
n_{0}$%
\begin{eqnarray}
\left\vert x_{n+1}-x_{n}\right\vert &<&\delta ,  \label{h1} \\
r_{n} &\geq &-\varepsilon \theta _{n}.  \label{hh1}
\end{eqnarray}%
Moreover, there exist two positive integers $n_{2}>n_{1}>n_{0}$ such that $%
x_{n_{1}}\in (w+\delta ,b)$ and $x_{n_{2}}\in (a,w-\delta ).$ Let $m$ be the
greatest integer between $n_{1}$ and $n_{2}$ such that $x_{m}\geq w-\delta .$
Therefore, in view of (\ref{h1}), we have $x_{m+1}<w-\delta \leq x_{m}<w.$
Now by going back to the dynamical system (\ref{PMp}) and using the fact that
$f(x_{m})\geq x_{m}+\varepsilon ,$ we get%
\begin{eqnarray*}
x_{m+1} &\geq &x_{m}+\theta _{m}\varepsilon +r_{m} \\
&\geq &x_{m}~\ \ (\text{ in view of (\ref{hh1})).}
\end{eqnarray*}%
This contradicts the fact that $x_{m+1}<x_{m}.$ Similarly, we can show that
the assumption $f(w)<w$ leads to a contradiction. We therefore conclude $%
f(w)=w$ for any $w\in (a,b)$ if $a<b.$ We continue working under the
assumption $a<b.$ Let $c=\frac{b+a}{2}$ and set $\alpha =\frac{b-a}{8}.$
There exists a positive integer $m_{0}$ such that
\begin{eqnarray}
\left\vert x_{m_{0}}-c\right\vert &<&\alpha ,  \label{h3} \\
\left\vert \sum_{n=m_{0}}^{m}r_{n}\right\vert &<&\alpha ,\forall m\geq m_{0}.
\label{h4}
\end{eqnarray}%
Let $m_{1}>m_{0}$ be the smallest positive integer such that $%
x_{m_{1}}\notin \lbrack c-2\alpha ,c+2\alpha ]$ (such integer exists since $%
[c-2\alpha ,c+2\alpha ]\subset (a,b)$).  Therefore, for any integer $n\in
[m_{0},m_{1}-1],$ we have $f(x_{n})=x_{n}$ and

\[
x_{n+1}=x_{n}+r_{n}.
\]%
Summing up the previous equalities, we obtain%
\[
x_{m_{1}}=x_{m_{0}}+\sum_{n=m_{0}}^{m_{1}-1}r_{n}.
\]%
Combining this identity with the estimations (\ref{h3}) and (\ref{h4}), we
get
\[
\left\vert x_{m_{1}}-c\right\vert \leq 2\alpha
\]%
which contradicts the definition of $x_{m_{1}}$. We then conclude that $a=b$. Hence, in view of Lemma \ref{l1}, we infer that the sequence $\{x_{n}\}$ converges in $\mathbb{R}$; let $x_{\infty }$ be its limit.  Clearly $x_{\infty }\in [0,1].$ Let us prove by contradiction that $f(x_{\infty })=x_{\infty
}.$ We suppose for instance that $f(x_{\infty })>x_{\infty }.$ Thanks to the
continuity of the function $f,$ there exists a positive real number $\beta
>0 $ and a positive integer $p_{0}$ such that, for any $n\geq p_{0},$
\begin{eqnarray*}
f(x_{n}) &\geq &x_{n}+2\beta , \\
r_{n} &\geq &-\beta\theta _{n}.
\end{eqnarray*}%
Hence, for any integer $n\geq p_{0},$%
\begin{eqnarray*}
x_{n+1}-x_{n} &=&\theta _{n}(f(x_{n})-x_{n})+r_{n} \\
&\geq &2\beta \theta _{n}+r_{n} \\
&\geq &\beta \theta _{n}.
\end{eqnarray*}%
Summing up these inequalities yields%
$$
\beta \sum_{n=p_{0}}^{\infty }\theta _{n} \leq x_{\infty }-x_{p_{0}}.
$$
This contradicts the assumption on the sequence $\{\theta _{n}\}.$ Similarly,
we can show that $f(x_{\infty })$ can not be less than $x_{\infty }.$ We
therefore conclude that $f(x_{\infty })=x_{\infty }$ which completes the
proof.
\end{proof}
\section{On the convergence of the perturbed continuous dynamical system (\ref{CMp})}
This section is devoted to the proof of Theorem \ref{Th2}. Before starting the proof of this theorem, let us recall the definition and
some simple proprieties of the omega limit set of a continuous real valued
function defined on the interval $[0,\infty ).$ (For more details on the notion of the omega limit in a more general context, we refer the readers to the books \cite{Hal}, \cite{Har}, and \cite{JH} .
\begin{definition}
let $u:[0,\infty )\to \mathbb{R}$ be a continuous function. The
omega limit set associated to the function $u$ (denoted by $\omega (u(t))$)
is the set of real numbers $z$ such that $\lim_{n\rightarrow \infty
}u(t_{n})=z$ for some positive sequence $\{t_{n}\}$ such that $t_n\rightarrow \infty$ as $n\rightarrow \infty$.
\end{definition}

The following lemma gathers some simple and well known properties of the omega limit
set associated to a continuous and bounded real valued function $u$ defined
on the interval $[0,\infty ).$

\begin{lemma}
\label{l2}
let $u:[0,\infty )\to \mathbb{R}$ be a continuous and bounded
function. Then
\begin{enumerate}
\item $\omega (u(t))=\cap _{k\in \mathbb{N}}\overline{u([k,\infty ))}^{\mathbb{R}%
},$ where $\overline{u([k,\infty ))}^{\mathbb{R}}$ denotes the closure of
the set $u([k,\infty ))$ in $\mathbb{R}.$
\item $\omega (u(t))$ is a nonempty compact interval of $\mathbb{R}.$
\item $u(t)$ converges to some real number $u_{\infty }$ as $%
t\rightarrow \infty $ if ad only if $\omega (u(t))=\{u_{\infty }\}.$
\end{enumerate}
\end{lemma}
The proof of Theorem \ref{Th2} uses also the following technical lemma which is inspired by
the papers  and \cite{CC} and \cite{JM}.

\begin{lemma}
\label{l3}
let $u:[0,\infty )\to \mathbb{R}$ be a continuous and bounded
function. If $\omega (u(t))=[a,b]$ with $a<b$ then for any real numbers $c$
and $d$ such that $a<c<d<b$ and any $T_{0}>0$ there exist four reals numbers
$s_{1},s_{2},\tau _{1}$ and $\tau _{2}$ such that:
\begin{enumerate}
\item $\tau _{i}>s_{i}>T_{0}$, $i=1,2.$
\item $u(t)\in [c,d]$ for any $t\in [s_{i},\tau _{i}],$ $%
i=1,2.$
\item $u(s_{1})=c$ and $u(\tau _{1})=d.$
\item $u(s_{2})=d$ and $u(\tau _{2})=c.$
\end{enumerate}
\end{lemma}
\begin{proof}
Let $T_{0}>0$ and $c$ and $d$ two real numbers such that $a<c<d<b.$ From the
definition of $\omega (u(t)),$ there exist two real numbers $s_{0}$ and $%
\tau _{0}$ such that $\tau _{0}>s_{0}>T_{0},$ $u(s_{0})<c$ and $u(\tau
_{0})>d.$ Let $s_{1}=\sup \{t\in [s_{0},\tau _{0}]:u(t)<c\}.$ Using
the continuity of $u$ we can easily verify that $s_{0}<s_{1}<\tau _{0},$ $%
u(s_{1})=c$ and $u(t)\geq c$ for any $t\in \lbrack s_{1},\tau _{0}]$. Let $%
\tau _{1}=\inf \{t\in [s_{1},\tau _{0}]:u(t)>d\}.$ Again the
continuity of $u$ ensures that $\tau _{0}>\tau _{1}>s_{1},u(\tau _{1})=d$
and $u(t)\leq d$ for every $t\in [s_{1},\tau _{1}].$ This completes
the construction of $s_{1}$ and $\tau _{1}.$ The construction of $s_{2}$ and
$\tau _{2}$ can be done similarly. In fact, there exist two real numbers $%
s_{0}^{\prime }$ and $\tau _{0}^{\prime }$ such that $\tau _{0}>s_{0}>T_{0},$
$u(\tau _{0}^{\prime })<c$ and $u(s_{0}^{\prime })>d.$ Now let $s_{2}=\sup
\{t\in [s_{0}^{\prime },\tau _{0}^{\prime }]:u(t)>d\}$ and $\tau
_{2}=\inf \{t\in [s_{2},\tau _{0}^{\prime }]:u(t)<c\}.$ Again by
using the continuity of $u$ we can easily verify that $s_{2}$ and $\tau _{2}$
satisfy all the required proprieties.
\end{proof}

Now we are in position to prove Theorem \ref{Th2}.

\begin{proof}
From Lemma \ref{l2}, there exist two real numbers $0\leq a\leq b\leq 1$ such that $\omega
(x(t))=[a,b].$ First, for the sake of a contradiction, we suppose that $a<b$.  Now let $w$ be an arbitrary element of the open interval $(a,b).$ Let us
suppose that $f(w)>w.$ Then there exist $\varepsilon ,\delta >0$ such that $%
[w-\delta ,w+\delta ]\subset (a,b)$ and $f(z)>z+\varepsilon $ for any $z\in [w-\delta ,w+\delta ].$ Let $T_{0}>0$ such that $r(t)\geq
-\varepsilon \theta (t),\forall t\geq T_{0}.$ According to Lemma \ref{l3}, there
exist $\tau _{2}>s_{2}>T_{0}$ such that $x(s_{2})=w+\delta ,~x(\tau
_{2})=w-\delta ,$ and $x(t)\in [w-\delta ,w+\delta ]$ for any $t\in
[s_{2},\tau _{2}].$ Therefore, for every $t\in [s_{2},\tau
_{2}],$%
\begin{eqnarray*}
x^{\prime }(t) &=&\theta (t)(f(x(t))-x(t))+r(t) \\
&\geq &\varepsilon \theta (t)+r(t) \\
&\geq &0.
\end{eqnarray*}%
Hence by integrating on the interval $[s_{2},\tau _{2}],$ we get the
contradiction $-2\delta \geq 0.$ Similarly, by using the times $s_{1}$ and $%
\tau _{1}$ instead of $s_{2}$ and $\tau _{2},$\ we can verify that the
assumption $f(w)<w$ leads also to a contradiction. We then conclude that if $%
a<b$ then $f(w)=w$ for any $w\in [a,b].$ We continuous working under
the assumption $a<b.$ Let $c=\frac{a+b}{2}$ and $\alpha =\frac{b-a}{4}.$ Let
$T_{1}>0$ be greater enough such that
\begin{equation}
\left\vert \int_{s}^{\tau }r(t)dt\right\vert \leq \alpha ,~\forall \tau
>s>T_{1}.  \label{h5}
\end{equation}%
According to Lemma \ref{l3}, there exist $\tau _{1}>s_{1}>T_{1}$ such that $%
x(s_{1})=c-\alpha ,~x(\tau _{1})=c+\alpha ,$ and $x(t)\in \lbrack c-\alpha
,c+\alpha ]$ for any $t\in \lbrack s_{1},\tau _{1}].$ Therefore, for any $%
t\in \lbrack s_{1},\tau _{1}],$

\begin{eqnarray*}
x^{\prime }(t) &=&\theta (t)(f(x(t))-x(t))+r(t) \\
&=&r(t).
\end{eqnarray*}%
Hence, by integrating the last equality between $s_{1}$ and $\tau _{1}$ we
get
\[
2\alpha =\int_{s_{1}}^{\tau _{1}}r(t)dt,
\]%
which contradicts (\ref{h5}). We therefore conclude that $a=b$ which means,
thanks to Lemma \ref{l2}, that $x(t)$ converges as $t\rightarrow \infty $ to some the
real number $x_{\infty }=a=b.$ It is clear that $x_{\infty }\in
\lbrack 0,1],$ let us show that it is a fixed point of $f.$ For the sake of
absurdity, we assume for instance that $f(x_{\infty })<x_{\infty }.$ The
continuity of $x(t)$ and $f$ and the assumption on $r(t)$ ensure the
existence of $\varepsilon >0$ and $T_{2}$ such that for any $t\geq T_{2}$%
\begin{eqnarray*}
x(t)-f(x(t)) &\geq &2\varepsilon , \\
r(t) &\leq &\varepsilon \theta (t).
\end{eqnarray*}%
Therefore, for every $t\geq T_{2},$
\begin{eqnarray*}
x^{\prime }(t) &=&\theta (t)(f(x(t))-x(t))+r(t) \\
&\leq &-\varepsilon \theta (t).
\end{eqnarray*}%
Integrating the last inequality between $T_{2}$ and $T>T_{2}$ and then letting $%
T\rightarrow \infty ,$ we get the inequality%
\[
\varepsilon \int_{T_{2}}^{\infty }\theta (t)dt\leq x(T_{2})-x_{\infty }
\]%
that contradict the assumption on $\theta .$ Similarly, we can show that
$f(x_{\infty })$ can not be greater than $x_{\infty }.$ We therefore
conclude that $f(x_{\infty })=x_{\infty }.$
\end{proof}
\section{A numerical study of the rate of convergence of the processes (\ref{PMp}) and (\ref{CMp})}
In the first part of this section, we numerically investigate  the effect of the speed of the vanishing of the sequence $\{\theta_n\}$ and the amplitude of the error term $\{r_n\}$ on the rate of convergence of the sequence $\{x_n\}$ generated by the perturbed process (\ref{PMp}) to a fixed point of $f$. We consider the case where the sequence $\{\theta_n\}$ is defined by $\theta_n=\frac{1}{(n+1)^\alpha}$
where $0<\alpha\leq1$, and the objective function $f:[0,1]\to [0,1]$ is given by:
\begin{equation}\label{FF}
f(x)=\left\{
\begin{array}{ll}
2(\frac{1}{4}-x), & 0\leq x\leq \frac{1}{4}, \\
4(x-\frac{1}{4}), & \frac{1}{4}\leq x\leq \frac{1}{2}, \\
4(\frac{3}{4}-x), & \frac{1}{2}\leq x\leq \frac{3}{4}, \\
2(x-\frac{3}{4}), & \frac{3}{4}\leq x\leq 1.%
\end{array}%
\right.
\end{equation}
\begin{figure}[ht]\label{grap1}
\centering
\includegraphics[width=8cm]{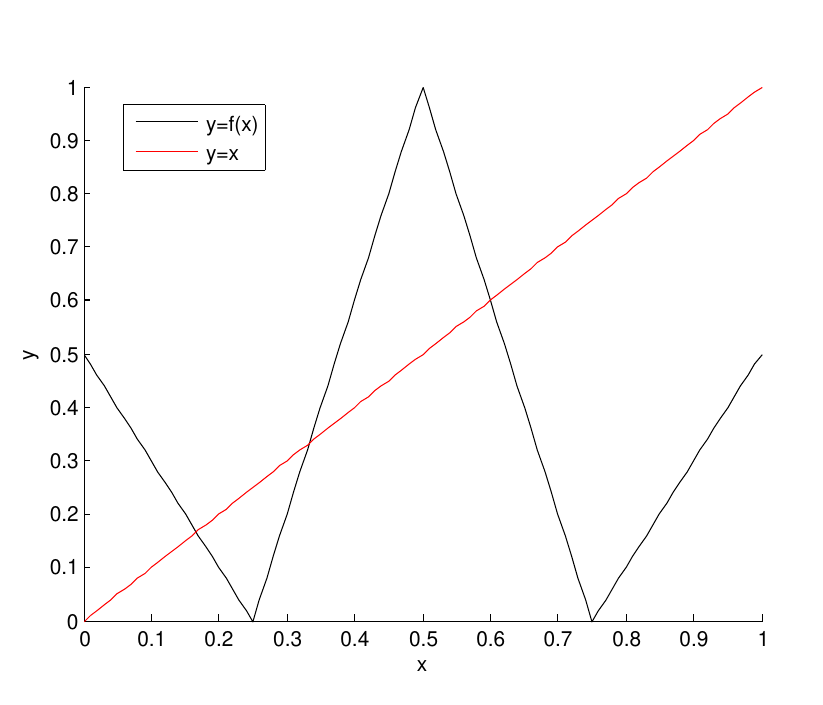}
\caption{The graph of the function $f$ and its fixed points}
\end{figure}
A simple calculation shows that $f$ has exactly three fixed points: $a=\frac{1}{6}$, $b=\frac{1}{3}$, and $c=\frac{3}{5}$ (see Figure 1). We can now introduce our example of the process (\ref{PMp}): Let $x_0$ be an arbitrary element of $[0,1]$. The sequence $\{x_n\}$ is generated by the iterative stochastic process
\begin{equation}\label{Eq1}
  x_{n+1}=\Pi((1-\theta_n)x_n+\theta_n f(x_n)+A \frac{M_n}{1+n^2}),
\end{equation}
where $ A>0$ is a constant and $\{M_n\}$ is a sequence of i.i.d random variables such that each $M_n$ follows the uniform distribution ${U}([-1,1])$. The function $\Pi:\mathbb{R}\rightarrow[0,1]$ is the metric projection on the interval $[0,1]$ which is defined explicitly by:
\begin{equation}\label{Op}
\Pi(x)=\left\{
\begin{array}{ll}
0, & x\leq 0, \\
x, & 0\leq x\leq 1, \\
1, & x\geq 1.
\end{array}%
\right.
\end{equation}
Clearly, the sequence $\{x_n\}$ belongs to the interval  $[0,1]$ and satisfies the process (\ref{PMp}) with an error term $r_n=\Pi(y_n+A\frac{M_n}{1+n^2})-y_n$ where $y_n=(1-\theta_n)x_n+\theta_n f(x_n)$. Since $y_n\in[0,1]$, we have $\Pi(y_n)=y_n.$ Then thanks to the fact that $\Pi$ is non expansive,
\[
\left\vert r_n\right\vert\leq \frac{A}{1+n^2}.
\]
Therefore, according to Theorem \ref{Th1}, the sequence $\{x_n\}$ converges to a fixed point $x_{\infty}$ of $f$. For the two chosen values of the amplitude of the error term $A=0.1$ and $0.001$ and some chosen values of $\alpha \in [0,1]$ and $\epsilon>0$, we conduct $K_{\text{max}}=100$ times the process (\ref{Eq1}) with arbitrary initial data $x_0$ in $[0,1]$ under the stopping criteria $ \left\vert f(x_n)-x_n\right\vert<\epsilon$. Let $N(A,\alpha,\epsilon)$ be the average of number of iterations needed to achieve the stopping criteria. The following two tables give  $N(0.1, \alpha,\epsilon)$ and $N(0.001, \alpha,\epsilon)$  for some values of $\alpha$ and $\epsilon$.

\begin{table}[h]
\centering
\begin{tabular}{ccccc}
\hline\hline
$\alpha/\epsilon $ & $\epsilon=0.1 $ & $\epsilon=0.01 $& $\epsilon=0.001 $& $\epsilon=0.0001 $\\ [0.5ex]
\hline
$\alpha=0.1$ & 4.27&77.07&117.41&133.64 \\
$\alpha=0.2$ & 5.83&24.40&36.46&41.72 \\
$\alpha=0.4$ & 5.05&12.83&17.03&22.50 \\
$\alpha=0.6$ & 4.04&7.65&11.84&20.54 \\
$\alpha=0.8$ & 3.53&6.47&11.20&21.63 \\
$\alpha=0.9$ & 3.65&5.83&10.73&25.07 \\
$\alpha=1$   & 3.40&6.14& 11.44& 28.89\\ [1ex]
\end{tabular}
\caption{$N(A=0.1,\alpha,\epsilon)$}
\end{table}
\begin{table}[h]
\centering
\begin{tabular}{ccccc}
\hline\hline
$\alpha/\epsilon $ & $\epsilon=0.1 $ & $\epsilon=0.01 $& $\epsilon=0.001 $& $\epsilon=0.0001 $\\ [0.5ex]
\hline
$\alpha=0.1$ & 4.88&76.14&119.15&134.84 \\
$\alpha=0.2$ & 5.14&22.94&35.09&38.39 \\
$\alpha=0.4$ & 4.97&11.52&16.49&17.81 \\
$\alpha=0.6$ & 3.95&7.39&9.49& 10.85\\
$\alpha=0.8$ & 3.46&5.87&6.68&8.63 \\
$\alpha=0.9$ & 3.57&5.72&6.70&9.20 \\
$\alpha=1$   & 3.28&4.89& 6.60& 8.78\\ [1ex]
\end{tabular}
\caption{$N(A=0.001,\alpha,\epsilon)$}
\end{table}
We notice that the process (\ref{PMp}) converges faster if $\alpha$ is close to $1$ (i.e., the sequence $\{\theta_n\}$ converges relatively quickly to zero) and $A$ is small (i.e., the perturbation term $\{r_n\}$ is relatively weak). In this case the speed of the convergence of the process (\ref{PMp}) is slightly better than the rate of the convergence of the classical bisection method applied for the function $f(x)-x$  for which it is known that  $N(\epsilon)$, the number of needed iterations to achieve the precision $\epsilon$,  is equal to the entire part of  $\frac{-\log(\epsilon)}{\log(2)}$ where $0<\epsilon<1$, \cite{Con}.
Now we study the effect of the amplitude $A$ of the error term  on the the limit of the sequences $\{x_n\}$ generated by the process (\ref{Eq1}). We take $\alpha=1$ and we consider the four cases corresponding to $x_0$ equals  $0$ or $1$ and $A$ equals $0.01$ or $0.5$. For each case we conduct the process (\ref{Eq1}) $N_{max}=1000$ times under the stopping criteria $\vert f(x_k)-x_k\vert<0.0001$. The results of these experiments are given in the following table where $p$ (respectively $q$ and $r$) represents the probability that the process converges to the fixed point $a=\frac{1}{6}$ (respectively $b=\frac{1}{3}$ and $c=0.6$) of $f$.
\begin{table}[h]
\centering
\begin{tabular}{cccc}
\hline\hline
$x_0, A$ & $p $ & q& $r $\\ [0.5ex]
\hline
$x_0=0, A=0.01$ & 0&0&1 \\
$x_0=0, A=0.5$ & 0.339&0&0.661 \\
$x_0=1, A=0.001$ & 0&0&1 \\
$x_0=1, A=0.5$ & 0.349&0&0.651\\ [1ex]
\end{tabular}
\caption{The effect of the stochastic perturbation term on the convergence of the process (\ref{Eq1})}
\end{table}
We observe that, in contrast to the deterministic bisection method, the process (\ref{Eq1}) with a significant stochastic perturbation term
$ e(t) $ is capable of identifying a broader set of fixed points of the objective function $f$.
\par The last part of this section is devoted to the numerical study of the asymptotic behavior of the trajectories $x(t)$ of a example of the continuous process (\ref{CMp}). Let $\theta$ be the function defined on $[0,\infty)$ by $\theta(t)=\frac{1}{(1+t)^\alpha}$ with $0<\alpha\leq1$ and the error term $e(t)$ given by $e(t)=A\frac{\sin(t)}{1+t^2}$ with $A>0$. We now consider the system:
\begin{equation}
\left\{
\begin{array}{l}
x^{\prime }(t)+x(t)=\Pi ((1-\theta (t))x(t)+\theta (t)f(x(t))+e(t)) \\
x(0)=x_{0}%
\end{array}%
\right.   \label{N1}
\end{equation}%
where $\Pi $ is given by (\ref{Op}) and  $f:[0,1]\to [0,1]$ is here a general
continuous function. We first prove the following existence result.
\begin{proposition}
For any initial data $x_{0}\in [0,1]$, the system has a solution $%
x:[0,\infty )\to [0,1]$ that satisfies in addition the
continuous process (\ref{CMp}) with an error term $r(t)$ satisfying
\[
\left\vert r(t)\right\vert \leq \frac{A}{1+t^{2}},~\forall t\geq 0.
\]
\end{proposition}

\begin{proof}
Since the function $F:\mathbb{R}\to \mathbb{R}$ defined by $%
F(x)=f(\Pi (x))$ is continuous, then from the classical Peano's existence theorem, the differential
equation
\[
x^{\prime }(t)+x(t)=\Pi ((1-\theta (t))x(t)+\theta (t)F(x(t))+e(t))
\]%
has at least one maximal solution $x:[0,T^{\ast })\to \mathbb{R}$
that satisfies the initial condition $x(0)=x_{0}.$ For any $t\in \lbrack
0,T^{\ast }),$%
\[
x(t)=e^{-t}x_{0}+\int_{0}^{t}e^{-(t-s)}h(s)ds,
\]%
with $h(s)=\Pi ((1-\theta (s))x(t)+\theta (s)F(x(t))+e(s)).$ Since $0\leq
h\leq 1,$ then for any $t\in \lbrack 0,T^{\ast }),$%
\[
0\leq \int_{0}^{t}e^{-(t-s)}h(s)ds\leq 1-e^{-t},
\]%
and as a consequence%
\[
0\leq x(t)\leq 1.
\]%
We therefore infer that the maximal solution $x(t)$ is in fact global (i.e., $%
T^{\ast }=\infty )$ and satisfies the system (\ref{N1}) since $F(x(t))=f(x(t))$ for any
$t\geq 0.$ Finally, if we set
\[
y(t)=(1-\theta (t))x(t)+\theta (t)f(x(t)),
\]%
then
\[
x^{\prime }(t)+\theta (t)x(t)=\theta (t)f((t))+r(t),
\]%
with%
\[
r(t)=\Pi (y(t)+e(t))-y(t).
\]%
Since $y(t)\in [0,1]$ and the projection operator $\Pi $ is non expansive,
$$ \left\vert r(t)\right\vert  \leq \left\vert e(t)\right\vert
\leq \frac{A}{1+t^{2}}.$$
\end{proof}
We now consider the particular case of the system (\ref{N1}) in which $x_0=\frac{1}{2}$ and $f$ is the function given by (\ref{FF}). In this case, since $f$ is a Lipschitz function, the solution $x(t)$ given by the previous Proposition is unique. According to Theorem \ref{Th2}, the trajectory $x(t)$ converges as $t\rightarrow\infty$ to a fixed point of $f$. The graphs (Figure 2) and (Figure 3) give an idea about the effect of the parameters $\alpha$ and $A$ on the speed of the convergence of $x(t)$.
\begin{figure}[ht]\label{grap2}
\centering
\includegraphics[width=8cm]{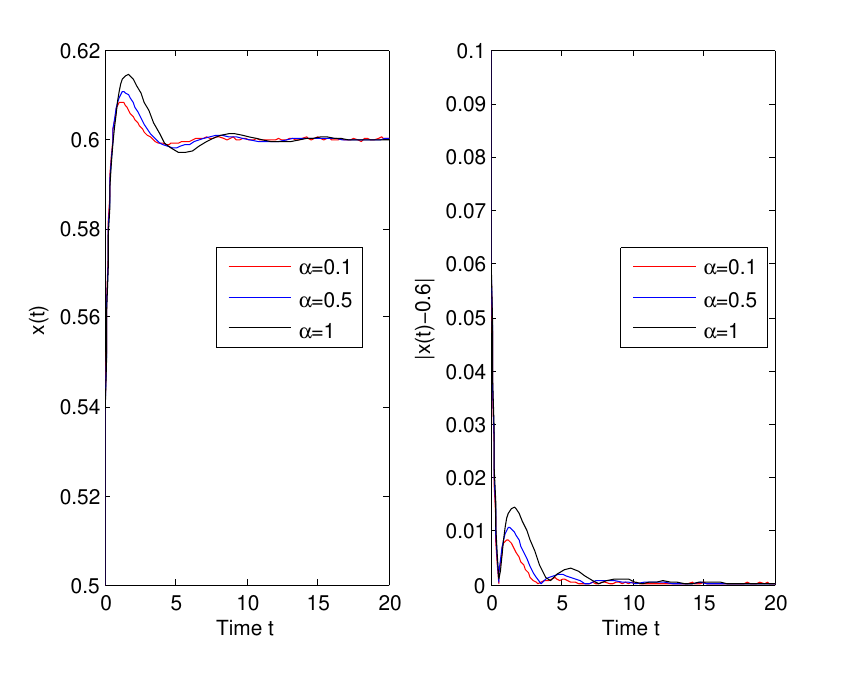}
\caption{Behavior of some trajectories $x(t)$ of the system (\ref{N1}) when  $A=0.1$ }
\end{figure}
\begin{figure}[ht]\label{grap3}
\centering
\includegraphics[width=8cm]{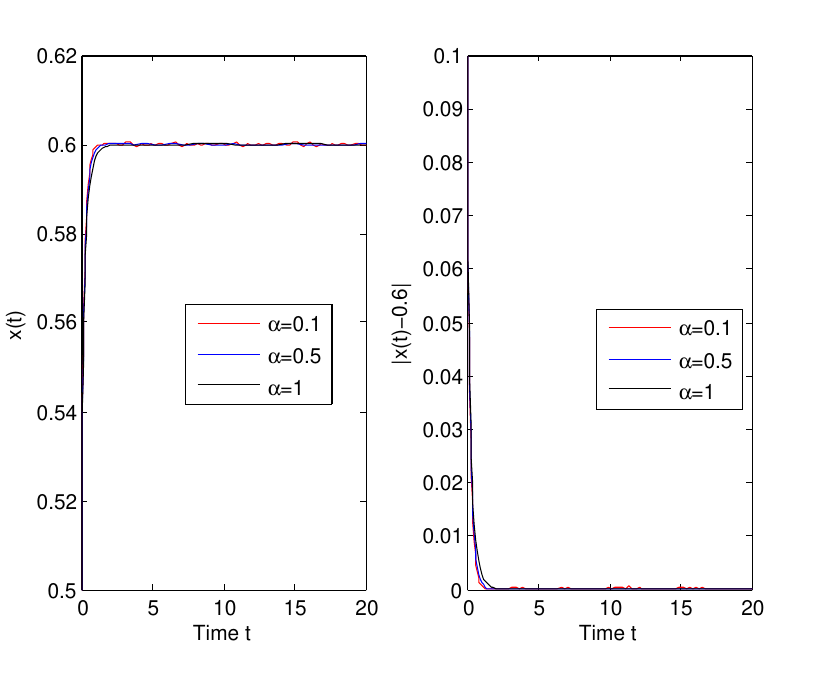}
\caption{Behavior of some trajectories $x(t)$ of the system (\ref{N1}) when  $A=0.001$ }
\end{figure}
We observe that the rate at which the function $\theta(t)$ vanishes as $t \to \infty$ has little to no effect on the convergence speed of the trajectories $x(t)$. Additionally, the trajectories tend to converge slightly faster when the error term $e(t)$ is sufficiently small.
\section{Appendix}
It seems that the original proof of Theorem \ref{Th01} due to Mann (see the proof of \cite[Theorem 4] {Man}) is incomplete. Hence, in this section, we provide a complete and detailed proof of this result.
\begin{proof}
The function $g(x)=f(x)-x$ defined on the interval $[0,1]$ is continuous,
satisfies $g(0)\leq 0$ and $g(1)\geq 0,$ and has $p$ as a unique root, then
from the intermediate value theorem, $g(x)>0$ on the interval $[0,p)$ and $%
g(x)<0$ on $(p,1].$ Now let $\varepsilon >0$ be an arbitrary but a fixed
real number. There exists $\delta >0$ such that
\begin{equation}
(x\in [0,1]\wedge \left\vert x-p\right\vert \geq \varepsilon
)\Rightarrow \left\vert g(x)\right\vert \geq \delta .  \label{m1}
\end{equation}%
On the other hand since,
\begin{eqnarray*}
\left\vert x_{n+1}-x_{n}\right\vert  &=&\frac{1}{n+1}\left\vert
g(x_{n})\right\vert  \\
&\leq &\frac{1}{n+1},
\end{eqnarray*}%
there exists a positive integer $n_{0}$ such that%
\begin{equation}
\left\vert x_{n+1}-x_{n}\right\vert <\varepsilon , \forall n\geq n_{0}.  \label{m2}
\end{equation}%
Now let us suppose that for any integer $n\geq n_{0},$ $x_{n}\notin [p-\varepsilon ,p+\varepsilon ].$ This implies that
either $x_{n}>p+\varepsilon $ for any $n\geq n_{0}$ or $x_{n}<p-\varepsilon $
for any $n\geq n_{0}$, because otherwise there exists $m\geq n_0$ such that $x_m>p+\varepsilon $ and  $x_{m+1}<p-\varepsilon $ which is impossible in view of (\ref{m2}). Let us assume for instance that $x_{n}>p+\varepsilon $
for any $n\geq n_{0}.$ Then, from (\ref{m1}) and the fact that $g(x)<0$ on $%
(p,1],$ we deduce that%
\[
x_{n+1}-x_{n}=\frac{1}{n+1}g(x_{n})\leq -\frac{\delta }{n+1},~\forall n\geq
n_{0}.
\]%
Therefore, for any $n>n_{0},$ we have the inequality%
\[
x_{n}-x_{n_{0}}\leq -\delta \sum_{k=n_{0}}^{n-1}\frac{1}{k+1}
\]%
which implies $\lim_{n\rightarrow \infty }x_{n}=-\infty $ contradicting the
fact that $\{x_{n}\}$ is bounded. We therefore conclude that there exists $%
n_{1}\geq n_{0}$ such that $x_{n_{1}}\in [p-\varepsilon
,p+\varepsilon ].$ We will now prove that for any $n\geq n_{0},$ if $x_{n}\in
[p-\varepsilon ,p+\varepsilon ]$ then $x_{n}\in [p-\varepsilon
,p+\varepsilon ]$. Let $n\geq n_{0}$ such that $x_{n}\in [
p-\varepsilon ,p+\varepsilon ].$ We consider the two possible cases:

\par\noindent The first case $x_{n}\in [0,p+\varepsilon ].$ We have $g(x_{n})\leq 0,
$ then $x_{n+1}=x_n+\theta_n g(x_n)\leq x_{n}.$ Combining this inequality with the fact that $%
\left\vert x_{n+1}-x_{n}\right\vert <\varepsilon ,$ we deduce that $%
x_{n+1}\in [p-\varepsilon ,p+\varepsilon ].$

\par\noindent The second case $x_{n}\in [p-\varepsilon ,0].$  We have $%
g(x_n)\geq 0,$ then $x_{n+1}\geq x_{n}.$ Hence, by using again the
fact $\left\vert x_{n+1}-x_{n}\right\vert <\varepsilon ,$ we also deduce
that $x_{n+1}\in [p-\varepsilon ,p+\varepsilon ].$ We deduce by induction that  for any $n\geq n_{1},x_{n}\in
[p-\varepsilon ,p+\varepsilon ].$ This means that $ x_n\rightarrow p$ as $n\rightarrow\infty$ and therefore completes the proof of the theorem.
\end{proof}

\end{document}